\documentclass[11pt]{article}
\usepackage{}
\usepackage[total={6in, 8.5in}]{geometry}
\usepackage{amsfonts}
\usepackage{amsthm}
\usepackage{amssymb}
\usepackage{amsmath}
\usepackage{enumerate}
\usepackage[
pdfauthor={ESYZ},
pdftitle={Toughness and spanning trees in K4mf graphs},
pdfstartview=XYZ,
bookmarks=true,
colorlinks=true,
linkcolor=blue,
urlcolor=blue,
citecolor=blue,
bookmarks=false,
linktocpage=true,
hyperindex=true
]{hyperref}
\usepackage{amsthm,amsmath,amssymb}

\usepackage{graphicx}

\usepackage{enumerate}

\usepackage{cases}
\theoremstyle{plain}
\newtheorem{theorem}{Theorem}
\newtheorem{lemma}[theorem]{Lemma}

\theoremstyle{definition}

\newtheorem{conjecture}[theorem]{Conjecture}

\newtheorem{question}[theorem]{Question}

\theoremstyle{remark}

\title{\bf Antimagic orientation of  graphs with  minimum degree at least 33}
\author{Songling Shan\\
	\small
	 Illinos State University,  Normal,IL 61790\\
	\small\tt sshan12@ilstu.edu\\}
\date{May 4, 2020}
\begin{document}

\maketitle

% E-JC papers must include an abstract. The abstract should consist of a
% succinct statement of background followed by a listing of the
% principal new results that are to be found in the paper. The abstract
% should be informative, clear, and as complete as possible. Phrases
% like "we investigate..." or "we study..." should be kept to a minimum
% in favor of "we prove that..."  or "we show that...".  Do not
% include equation numbers, unexpanded citations (such as "[23]"), or
% any other references to things in the paper that are not defined in
% the abstract. The abstract will be distributed without the rest of the
% paper so it must be entirely self-contained.
\begin{abstract}
	An antimagic labeling
	of a directed graph $D$ with $n$ vertices and $m$ arcs is a bijection from the set of arcs of
	$D$ to the integers $\{1, \cdots,  m\}$ such that all $n$ oriented vertex sums are pairwise distinct,
	where an oriented vertex sum is the sum of labels of all arcs entering that vertex minus the
	sum of labels of all arcs leaving it. A  graph $G$ has an  antimagic orientation if  it has  an orientation which  admits an antimagic
	labeling. Hefetz, M{\"{u}}tze, and Schwartz  conjectured that
	every connected  graph admits an antimagic orientation.
	In this paper, we show that every bipartite graph  without both isolated and degree 2 vertices admits an antimagic orientation and every  graph $G$ with $\delta(G)\ge 33$ 
	admits an  antimagic orientation. Our proof relies on a newly developed structural property of bipartite graphs, which might be of independent interest.  
	
\end{abstract}

\section{Introduction}
All graphs considered are simple and finite unless otherwise stated. 
For two integers  $p,q$,  $[p,q]:=\{p,p+1\ldots, q\}$ if $q\ge p$,
and  $[p,q]:=\emptyset$ if $q<p$. 
A {\it labeling\/} of a graph $G$  with $m$ edges is a bijection from $E(G)$ to a set $S$ of $m$ integers,
and the {\it vertex sum\/} at a vertex $v\in V(G)$ is the sum of labels on the edges incident to $v$.
A  labeling  is {\it antimagic\/} if  $S=[1,m]$ and all the vertex sums are distinct. 
A graph is {\it antimagic\/} if it has an antimagic labeling.

Hartsfield and Ringel~\cite{MR1282717} introduced antimagic labelings in 1990 and conjectured
that every connected graph other than $K_2$ is antimagic.
There have been some significant progress towards this conjecture. Let $G$ be a graph with $n$ vertices
other than $K_2$.
In 2004, Alon, Kaplan, Lev, Roditty, and Yuster~\cite{MR2096791} showed that
there exists a constant $c$ such that if $G$ has minimum degree
at least $ c \cdot \log n$,
then $G$ is antimagic.  They also proved that $G$ is antimagic when the
maximum degree of $G$ is at least $n-2$, and they
proved that all complete multipartite graphs (other than $K_2$) are antimagic.
The latter result of Alon et al. was improved by
Yilma~\cite{MR3021347} in 2013.

Apart from the  results above on dense graphs,
the antimagic labeling conjecture has been also verified for
regular graphs.
Started with Cranston~\cite{MR2478227} showing that
every bipartite regular graph is antimagic,
regular graphs of odd degree~\cite{MR3372337},
and finally all regular graphs~\cite{MR3414180,MR3515572} were  shown
to be antimatic sequentially.
For more results on the antimagic labeling conjecture
for other classes of graphs, see~\cite{MR3527991, MR2174213, MR2682515,MR2510327}.

Hefetz, M{\"{u}}tze, and Schwartz~\cite{MR2674494} introduced the
variation of antimagic labelings, i.e., antimagic labelings on
directed graphs. An {\it antimagic\/} labeling of a directed graph with  $m$ arcs is
a bijection from the set of arcs  to the integers $\{1,...,m\}$ such that
any two oriented vertex sums are pairwise distinct, where an {\it oriented
	vertex sum\/} is the sum of labels of all arcs entering that vertex minus
the sum of labels of all arcs leaving it.
A digraph is called {\it antimagic\/} if it admits an antimagic labeling. For an undirected graph $G$, if it has an orientation such that
the orientation is antimagic, then we say $G$ admits an {\it antimagic orientation\/}.
Hefetz, M{\"{u}}tze, and Schwartz in the same paper posted the following problems.

\begin{question}[\cite{MR2674494}]\label{question1}
	Is every connected directed graph with at least 4 vertices antimagic?
\end{question}

\begin{conjecture}[\cite{MR2674494}]\label{antimagic-orientation}
	Every connected  graph admits an antimagic orientation.
\end{conjecture}

Hefetz, M{\"{u}}tze, and Schwartz~\cite{MR2674494} showed
that every orientation of a dense graph is antimagic
and almost all regular graphs
have an antimagic orientation. Particulary,
they showed that every orientation of stars\,(other than $K_{1,2}$), wheels,
and complete graphs\,(other than $K_3$)
is antimagic. 
Conjecture \ref{antimagic-orientation} has been also  verified  for regular graphs \cite{MR2674494,1707.03507, SH2019,Y2019}, biregular bipartite graphs with minimum degree at least two \cite{SY2017}, Halin graphs \cite{YCZ2019},  graphs with large maximum degree \cite{YCOP2019},
and graphs with large independence number~\cite{SYZ}.  
 In this paper, by supporting
Conjecture~\ref{antimagic-orientation}, we obtain the
results below.

\begin{theorem}\label{th1}
	Every bipartite graph  with no vertex of degree 0 or 2    admits an antimagic orientation. 
\end{theorem}

\begin{theorem}\label{th2}
	Every    graph $G$ with $\delta(G)\ge 33$ admits an antimagic orientation. 
\end{theorem}

The remainder of this paper is organized as follows. 
We introduce several preliminary results in Section 2. In Section 3, we prove Theorem~\ref{th1},
and  in Section 4, we prove Theorem~\ref{th2}.

\section{Notation and Preliminary Lemmas}

Let $G$ be a graph.  We use  $e(G)$ for $|E(G)|$. 
For $S\subseteq V(G)$, $G[S]$ is the subgraph of $G$ induced by $S$. 
For two disjoint subsets $S,T\subseteq V(G)$, 
we denote by $E_G(S,T)$ the set of edges in $G$ with one endvertex in $S$
and the other in $T$, and let $e_G(S,T)=|E_G(S,T)|$. 
If $G$ is bipartite with partite sets $X$ and $Y$,
we denote $G$ by $G[X,Y]$ to emphasis the bipartitions.
Given an orientation $D$ of $G$, a labeling $\sigma$ on $A(D)$ that is the set of arcs of $D$, 
and a vertex $v\in V(D)$, we use $s_{[D,\sigma]}(v)$ to denote 
the {\it oriented sum at $v$\/} in $D$, which  is the sum of labels of all arcs entering $v$ minus
the sum of labels of all arcs leaving it in the digraph $D$.

For a matching $M$ of $G$, we use $V(M)$ to denote the set of vertices saturated by $M$. For a vertex $x\in V(M)$, $M(x)$
is the vertex that is matched to $x$ in $M$. 
For each subset $X\subseteq V(M)$, if  $X$ is an independent set in $M$, then $M(X)$
is the set of of vertices that are matched to vertices from $X$
in $M$. By this definition, $|X|=|M(X)|$ and $X$ and $M(X)$ are disjoint. 
An {\it $M$-augmenting path } is a path whose edges are alternating between edges in $M$ and edges not in $M$ and with both endpoints being not saturated by $M$.

A \emph{trail} is an alternating sequence of vertices and edges $v_0e_1v_1\ldots e_tv_t$ such that
$v_{i-1}$ and $v_i$ are the endvertices of $e_i$, for each $i\in [1,t]$, and the edges are all distinct (but there might be repetitions among the vertices).  A trail is \emph{closed} if $v_0= v_t$, and is \emph{open} otherwise. An {\it Euler tour} of $G$ 
is a closed trail in $G$ that contains all the edges of $G$. 
We will need the following classic result of Euler in proving  a lemma later on. 
\begin{theorem}[Euler, 1736]\label{Euler}
	A mutigraph $G$ has an Euler tour if and only if 
	$G$ has at most one nontrivial component and every vertex of $G$
	has an even degree. 
\end{theorem}

\begin{lemma}[ \cite{MR2510327}]   \label{partitions}
	Let $t, n$ be integers with $t\ge 1$ and $n\ge 2$, and let 
	$n=r_1+\ldots+r_t$ be a partition of $n$, where  
	$r_i $  is an integer that is at least 2 for each $i\in [1,t]$. 
	Then the set $\{1,\ldots, n\}$ can be partitioned into 
	pairwise disjoint subsets $R_1, \ldots, R_t$
	such that for each $i\in [1,t]$, $|R_i|=r_i$
	and $\sum_{r\in R_i} r \equiv 0  \pmod{n+1}$ if $n$ is even, and 
	$\sum_{r\in R_i} r \equiv 0 \pmod{n}$ if $n$ is odd. 
\end{lemma}

The following result was proved in~\cite{SYZ} without the furthermore part. 
However, the furthermore part is easy to obtain by following the same proof of Lemma 2.2 in 
\cite{SYZ} by just letting vertices in $T$ to be not the endvertices of 
the edge-disjoint trails that decompose $E(G)$, which can be definitely guaranteed
by the conditions imposed on  $T$.  
So we omit the proof. 

\begin{lemma}\label{consecutive}
Let $p, m $ be  integers with with $p\ge 0, m\ge 1$, and let $G$ be a graph with $m$ edges. Then there exist an
orientation $D$ of $G$ and a bijections $\sigma: A(D)\rightarrow \{p+1,\ldots, p+m\}$ such that for each $v\in V(G)$,  
\begin{eqnarray*}
%-\lfloor \frac{d(v)-1}{2}\rfloor -(p+m)\le &  s(D,\sigma_1)(v)&\le -\lfloor \frac{d(v)-1}{2}\rfloor +(p+m), \quad \text{and }\\
 -(p+m)+\big\lfloor \frac{d_G(v)-1}{2}\big\rfloor\le &  s_{[D,\sigma]}(v)&\le \big\lfloor \frac{d_G(v)-1}{2}\big\rfloor +(p+m). 
\end{eqnarray*}
Furthermore, for $T\subseteq V(G)$  and each $v\in T$, if $d_G(v)$ is even and $N_{G}(v)\cap (V(G)\setminus T)\ne \emptyset$, 
then we can choose $\sigma$ so that 
%$s(D,\sigma_1)(v)=-\frac{d_G(v)}{2},  \quad 
$s_{[D,\sigma]}(v)=\frac{d_G(v)}{2}$. 
\end{lemma}

%\begin{lemma}\label{partitions2}
%	Let $t, n$ be integers with $t\ge 1$ and $n\ge 2$, and let 
%	$n=r_1+\ldots+r_t$ be a partition of $n$, where  
%	$r_i $  is an integer that is at least 2 for each $i\in [1,t]$. 
%	Then the set $\{1,\ldots, n\}$ can be partitioned into 
%	pairwise disjoint subsets $R_1, \ldots, R_t$
%	such that for each $i\in [1,t]$, $|R_i|=r_i$
%	and $\sum_{r\in R_i} r \equiv 0 (\md n+1)$ if $n$ is even, and 
%	$\sum_{r\in R_i} r \equiv 0 (\md n)$ if $n$ is odd. 
%\end{lemma}

% Lemma 2.3 ([10]) Let t ≥ 1 and n ≥ 2 be integers and let n = r1 + · · · + rt be a partition of n,
%where ri ≥ 2 is an integer for all i ∈ [t]. Then the set {1, 2, . . . , n} can be partitioned into pairwise
%6
%disjoint subsets R1, . . . , Rt such that for all i ∈ [t], |Ri
%| = ri
%, and P
%r∈Ri
%r ≡ 0 (mod n+1) if n is
%even and P
%r∈Ri
%r ≡ 0 (mod n) if n is odd.

\begin{lemma}\label{Teven}
	Let $p, m $ be  integers with $p\ge 0$ and $m\ge 1$, and let $G[S,T]$ be a bipartite graph with $m$ edges such that every vertex from $T$ has an even degree in $G$ (so $m$ is even). If $m\equiv 0 \pmod{4}$, let $\delta_m=p+m$; and if $m\equiv 2 \pmod{4}$,
	 let $\delta_m=p+m+1$. 
	Then there exist an
	orientation $D$ of $G$ and a bijection $\sigma: A(D)\rightarrow \{p+1,\ldots, p+m-1\}\cup\{\delta_m\}$ such that 
	\begin{eqnarray*}
		 & s_{[D,\sigma]}(v)&=
		  -d_G(v) \quad \quad\quad \quad\,\,\,\,
		   \text{for each $v\in T$, and }\\
	\big\lfloor \frac{d_G(v)-1}{2}\big\rfloor -\delta_m\le & s_{[D,\sigma]}(v)&\le \big\lfloor \frac{d_G(v)-1}{2}\big\rfloor +\delta_m \quad \text{for each $v\in S$}. 
	\end{eqnarray*}
\end{lemma}	

\begin{proof}
Suppose $G$ has in total $2\ell$ vertices of odd degree for some integer $\ell\ge 0$.  We obtain a new graph $G^*$ by pairing up  these vertices  into $\ell$
pairs,  and for each pair,   adding an edge joining the two vertices. 
Note that $G^*=G$ if $\ell=0$. 

 Each component of $G^*$ has  an Euler tour by  Theorem~\ref{Euler}. 
 By deleting all the edges in $E(G^*)\setminus E(G)$, we partition 
 all edges of $G$ into $\ell$ trails $T_1, T_2, \ldots, T_\ell$  (each $T_i$ is either open or closed).   For each $i\in[1,\ell]$,  let 
 $$
T_i=x_{t_{i-1}+1} e_{t_{i-1}+1}y_{t_{i-1}+1}  f_{t_{i-1}+1}x_{t_{i-1}+2} \ldots y_{t_i-1}f_{t_i-1}x_{t_i},  
 $$
 where $t_0=0$. Note that  $t_m-1=m/2$  and $|E(T_i)|=2(t_i-1-t_{i-1})$. 
 Since for every $v\in T$, $d_G(v)$ is even, we 
 can further assume that in each $T_i$,  for each $j\in [t_{i-1}+1, t_i]$,  
 $$x_j\in S \quad \text{and} \quad y_j\in T.$$
 Also by the construction of $T_i$'s, each vertex from $S$
 is the endvertex of at most one open trail. 
 
% In other words, $E(G)=\{e_1, e_2, \ldots, e_{m}\}$
% and for each $j\in [1,m]$, 
% $$e_j=
% \begin{cases}
% s_jt_j, & \text{if $j$ is odd};\\ 
% t_js_{j+1}, & \text{if $j$ is even}. 
% \end{cases}
% $$
 For each $j\in [1,m/2-1]$,  we direct each edge 
 $e_j$  from $x_j$ to $y_j$,  and direct each edge $f_j$ from $y_j$ to $x_{j+1}$.  
 Denote by $D$ this  orientation of $G$. 

If $m\equiv 0 \pmod{4}$,  for each $i\in [1,m/4]$,  let 
\begin{eqnarray*}
\sigma(e_{2i-1})=4i-3,\quad  &\sigma(f_{2i-1})&=4i-1;\\
\sigma(e_{2i})=4i-2,   \quad &\sigma(f_{2i})&=4i.\\
\end{eqnarray*}
If $m\equiv 2 \pmod{4}$, let 
\begin{eqnarray*}
	\sigma(e_{2i-1})=4i-3,\quad  &\sigma(f_{2i-1})&=4i-1 \quad \text{for each $i\in [1,\frac{m+2}{4}]$};\\
	\sigma(e_{2i})=4i-2,   \quad &\sigma(f_{2i})&=4i\quad \text{for each $i\in [1,\frac{m-2}{4}]$}.\\
\end{eqnarray*}
By the definition of $\sigma$ above, for each $j\in [1,m/2]$,  $e_j,f_j$
contributes $-2$ to the vertex sum at $y_j$  that is shared by $e_j$ and $f_j$. 
Since  for each vertex $y\in T$, the edges incident to $y$ in $G$
are partitioned into $\frac{d_G(y)}{ 2}$ pairs of edges in the form of 
$e_j, f_j$,  it holds $s_{[D,\sigma]}(y)=
-d_G(y)$. 

For each $j\in [1,m/2-1]$,  $f_j,e_{j+1}$
contributes $1$ to the vertex sum at $x_{j+1}$  that is shared by $f_j$ and $e_{j+1}$. 
For each vertex $x\in S$, the edges incident to $x$ in $G$
are partitioned into  at least $\lfloor\frac{d_G(x)-1}{ 2}\rfloor$ pairs of edges in the form of 
$f_j, e_{j+1}$.  If $d_G(x)$ is odd, then $x$
is the endvertex of exactly  one open trails in $\{T_1, T_2, \ldots, T_\ell\}$. 
Thus, the edge  incident to $x$ not counted in the pairs $f_j, e_{j+1}$ 
has a label in  $[-\delta_m, \delta_m]$. 
If $d_G(x)$ is even, then $x$
can be the endvertices  of at most one closed trails in $\{T_1, T_2, \ldots, T_\ell\}$.
Thus, the two edges  incident to $x$ not counted in the pairs $f_j, e_{j+1}$ 
have a label in  $[-\delta_m, \delta_m]$:   one is negative and the other is 
positive, which add up to a value in  $[-\delta_m, \delta_m]$.  Hence,  for each $x\in S$,  it holds 
 $\lfloor\frac{d_G(x)-1}{2}\rfloor -\delta_m\le  s_{[D,\sigma]}(x)\le \lfloor \frac{d_G(x)-1}{2}\rfloor +\delta_m$. 
This finishes the proof of Lemma~\ref{Teven}. 
\end{proof}

The following result on bipartite graphs is heavily used in our proofs,  which  might be 
of independent interest to other applications also.  
\begin{lemma}\label{matching-in-bipartite}
	If $G$ is a bipartite graph, then $V(G)$ has a partition $S\cup T$
that satisfies the following conditions:
	\begin{enumerate}[(a)]
		\item $G$ has a matching  $M$ with $M\subseteq E_G(S,T)$ and $M$ saturates $S$; 
%		\item $M$ 
%		saturates $S$;
		\item $T$ is an independent set in $G$.
	\end{enumerate}
	\end{lemma}

\begin{proof}
	It suffices to prove the statement only for every component of $G$. 
	Thus we may assume that $G$ is connected. 
	Let $[X,Y]$ be a bipartition of $G$. Assume, without loss of generality, that 
	$|X|\le |Y|$. 
	Let $M$ be a matching of $G$ that saturates the largest number of vertices from $X$. We will find a desired partition $S\cup T$ of $V(G)$
	based on $X$ and $Y$.
	
	If $X\setminus V(M)=\emptyset$, then we are done by letting $S=X$ and $T=Y$. Thus, $X\setminus V(M)\ne \emptyset$.  Let 
	\begin{align*}
X_1&=X\cap V(M),  \quad &X_0&=X\setminus X_1, \\
Y_1&=Y\cap V(M), \quad &Y_0&=Y\setminus Y_1.
	\end{align*}
Since $|X\cap V(M)|=|Y\cap V(M)|$ and $|Y|\ge |X|$, $X_0\ne \emptyset$ implies  $Y_0\ne \emptyset$.   By the maximality of $M$, it holds 
\begin{equation}\label{noedge1}
E_G(X_0,Y_0)=\emptyset. 
\end{equation}
Let 
	\begin{align*}
B_0&=N_G(X_0), & A_0&=M(B_0), & C_0&=N_G(Y_0),  & D_0&=M(C_0). 
\end{align*}
Clearly, $B_0, C_0\ne \emptyset$ as $G$ is connected. 
For each integer $i$ with $i\ge 1$, define 
$$
B_i=N_G(A_{i-1})\setminus \bigcup_{j=0}^{i-1}B_j, \quad A_i=M(B_i). 
$$
Let
\[
B=\cup_{i=0}^\infty B_i, \quad A=\cup_{i=0}^\infty A_i. 
\]
By the definition, $B_i\cap B_j=\emptyset$ and $A_i\cap A_j=\emptyset$
for every pair of $i,j$ with $i,j\ge 0$ and $i\ne j$. 
Since $|A_i|=|B_i|$ for each $i$ with $i\ge 0$, it holds 
\begin{equation}\label{AB}
|A|=|B|.
\end{equation}

Let $$X_r=X_1\setminus A, \quad Y_r=Y_1\setminus B.$$
By the definition of $A$, 
\begin{equation}\label{AYr}
E_G(A,Y_r)=\emptyset.  
\end{equation}
Let 
$$
S=B\cup X_r, \quad T=A\cup Y_r\cup X_0\cup Y_0. 
$$

It is left to show that $S\cup T$ is a desired partition of $V(G)$. 
Since $|A|=|B|$ by \eqref{AB}, $|S|=|X_1|$. 
Furthermore, by the definitions of $A$ and $B$, $A=M(B)$, and consequently $X_r=M(Y_r)$. Thus, 
 $M$ is still a matching in $G$ that saturates $S$ and has size $|S|$, 
 and $M\subseteq E_G(S,T)$. 
We only show that $T$ is an independent in $G$. As each of  $A, Y_r$, $X_0$
and $Y_0$ is an independent set in $G$, \eqref{noedge1} and \eqref{AYr}, respectively,  
implies  that $A\cup Y_r$ and $X_0\cup Y_0$ 
are independent sets in $G$.   
Since $E_G(X_0, A)=\emptyset$  and $E_G(X_0, Y_r)=\emptyset$ by $N_G(X_0)=B_0\subseteq B$,  $A\cup Y_r\cup X_0$ 
is an independent set in $G$. 
Since $E_G(Y_0,Y_r)=\emptyset$ and $E_G(Y_0, X_0)=\emptyset$ by~ \eqref{noedge1}, we are only left to show that  $E_G(Y_0,A)=\emptyset$.

It suffices to only show that $D_0\subseteq Y_r$.  Since  $D_0\subseteq Y_r$ implies that $C_0\subseteq X_r$ by the definitions 
of the sets $A$ and $B$, and   $C_0\subseteq X_r$ 
implies that $C_0\cap A=\emptyset$, which yields $E_G(Y_0, A)=\emptyset$.  

To show $D_0\subseteq Y_r$, we just show that 
for each $i$ with $i\ge 0$,  $E_G(A_i, D_0)=\emptyset$. 
Assume to the contrary and let $k$ be the  smallest index such that 
$E_G(A_k, D_0)\ne \emptyset$.
Let $d_0\in D_0$ and $a_k\in A_k$ such that $d_0a_k\in E(G)$, $b_k= M(a_k)$, 
$a_{k-1}\in A_{k-1}$ such that $a_{k-1}b_k\in E(G)$.
In general, for each $i=k-1, k-2,\ldots, 1$, let 
$$
b_i=M(a_i), \quad a_{i-1}\in A_{i-1} \quad \text{such that $a_{i-1}b_i\in E(G)$}. 
$$ 
Furthermore, let $b_0=M(a_0)$ and $x_0\in X_0$
such that $b_0x_0\in E(G)$, and $c_0=M(d_0)$ and $y_0\in Y_0$
such that $c_0y_0\in E(G)$. 

Note that for $i,j\in [0,k]$ with $i\ne j$, $a_i\ne a_j$
and $b_i\ne b_j$, as $A_i\cap A_j=\emptyset$ and $B_i\cap B_j = \emptyset$. 
Furthermore, by the minimality of $k$, $a_i\ne d_0$ and $b_i\ne d_0$.  
Thus  
$$P:=y_0c_0d_0a_kb_ka_{k-1}b_{k-1}\ldots b_1a_0b_0x_0$$
is an $M$-augmenting path, and $M':=(M\setminus E(P))\cup (E(P)\setminus M)$
is a matching in $G$ such that 
$|V(M')\cap X|>|V(M)\cap X|$, showing a contradiction to the choice of $M$.
Therefore, $E_G(A_i, D_0)=\emptyset$ for each $i$ with $i\ge 0$. 
This  completes the proof.  
\end{proof}
	
	\section{Proof of Theorem~\ref{th1}}
	
	Let  $S\cup T$ be  a partition of $V(G)$
	satisfying the requirements in Lemma~\ref{matching-in-bipartite}.
	Let 
	$$
	n_1=|S|, \quad n_2=|T|, \quad S=\{x_1, x_2, \ldots, x_{n_1}\}, \quad 
	\quad   T=\{y_1, y_2, \ldots, y_{n_2}\}. 
	$$
	
	Assume, without loss of generality, that  
	$$M=\{x_1y_1, x_2y_2, \ldots, x_{n_1}y_{n_1}\}.$$.
	For each $i\in [n_1+1, n_2]$, let $e_i$ be an edge incident to $y_i$ 
	in $G$, and let 
	$$
	M^*=M\cup \{e_{n_1+1}, \ldots, e_{n_2}\}.
	$$
	In other words,  each vertex from  $T$ is incident to one and exactly one edge from $M^*$.  
	Furthermore, let 
	$$
	H=G-E(G[S])-M^*, \quad m_1=|E(G[S])|, \quad m_2=e(H). 
	$$
	Clearly, $m_1+m_2+|M^*|=m_1+m_2+n_2=m:=e(G)$.
	Let $T_1=\{y\in Y: d_G(y)=1\}$ and $t_1=|T_1|$.  Assume, without loss of generality, that 
	$T_1=\{y_{n_2-t_1+1},y_{n_2-t_1+2}, \ldots, y_{n_2}\}$. 
	Clearly,  $d_H(y_i)=0$ for each $i\in [n_2-t+1,n_2]$. 
	We consider two cases below regarding how large $n_2$ is. 
	
	\medskip 
	
	{\bf \noindent Case 1: $n_2\le m_2$}. 
	
	\medskip 
	This case basically follows the same idea as  in the Proof of Theorem 1.5 in~\cite{SYZ}, but we repeat the process for self-completeness. 
	
 We give an orientation  $D$ of $G$ and a labeling $\sigma$ of $A(D)$ through four parts below. 
\begin{enumerate}[(1)]
	\item Orient and label $H$: direct each edge from $T$ to $S$. 
	For each $i\in [1, n_2-t_1]$, let $A_i$ be the set of all edges incident to $y_i$ in $H$. 
	Clearly,  $|A_1|+|A_2|+\ldots+|A_{n_2-t_1}|=m_2$. 
	Since $G$ has no vertex of degree 2 or isolated vertex,  $|A_i|\ge 2$. 
	By applying Lemma~\ref{partitions} to $m_2$ with 
	$t=n_2-t_1$ and $r_i=|A_i|$ for each $i\in[1,t]$, the set $\{1,2,\ldots, m_2\}$
	can be partitioned into $R_1, R_2, \ldots, R_{n_2-t_1}$ such that for each $i\in [1,n_2-t_1]$, 
	$|R_i|=|A_i|$ and $\sum_{r\in R_i} r \equiv 0 \pmod{m_2}$ if $m_2$ is even, and 
	$\sum_{r\in R_i} r \equiv 0 \pmod{m_2}$ if $m_2$ is odd.  Label edges in $A_i$
	by integers in $R_i$ in an arbitrary way as long as distinct edges receive distinct labels. 
	\item  Orient and label $G[S]$:  applying Lemma~\ref{consecutive} to get the orientation and labeling with $p=m_2$ and $m=m_1$; 
	\item  Orient and label $M^*\setminus M=\{e_{n_1+1}, \ldots, e_{n_2}\}$: direct each edge from $T$ to $S$, 
	and for each $i\in [n_1+1, n_2]$, assign $m_1+m_2+(i-n_1)$ to $e_i$. 
	
	\medskip
	
	Let $D^*$ be the union of the digraphs obtained through the three  parts above, and $\sigma^*$ be the  labeling on $A(D^*)$ consists of the  three labelings above.
	Assume that the sums at vertices from $S=\{x_1, \ldots, x_{n_1}\}$ satisfy 
	$$
	s_{[D^*,\sigma^*]}(x_1)\le	s_{[D^*,\sigma^*]}(x_2)\le \ldots \le s_{[D^*,\sigma^*]}(x_{n_1}). 
	$$
	
	\medskip

	\item  Orient and label  $M$: direct each edge from $T$ to $S$,   and for each $i\in [1, n_1]$, assign $m_1+m_2+n_2-n_1+i$ to $x_iy_i$.  
	
	 Let $D$ and $\sigma$ be the resulting orientation  and labeling, respectively. 
	It is clear that $\sigma$ is injective. We show that $\sigma$ 
	is an antimagic labeling of $D$. 
	By Step 4, we have 
	$$
	s_{[D,\sigma]}(x_1)<s_{[D,\sigma]}(x_2)<\ldots <s_{[D,\sigma]}(x_{n_1}). 
	$$
	Furthermore, for each $i\in [1,n_1]$, by Step 2,   $s_{[D^*,\sigma^*]}(x_i)\ge \lfloor\frac{d_{G[S]}(x_i)-1}{2}\rfloor -m_1-m_2$,  by Steps 3 and 4, we know $s_{[D,\sigma]}(x_i)\ge s_{[D^*,\sigma^*]}(x_i)+m_1+m_2+n_2-n_1+i>0$. 
	For each vertex $y_i\in T$, $i\in [1,n_2]$, all the edges incident to $y_i$
	are oriented towards $S$. Thus, $s_{[D,\sigma]}(y_i)<0$. 
	
	Thus, for each $x\in S$ and each $y\in T$, 
	$s_{[D,\sigma]}(x)>s_{[D,\sigma]}(y)$. 
	Therefore, it is left to only show that all vertices from $T$
have distinct sums under $\sigma$ in $D$. 

By Steps 1, 3 and 4,  for each $i\in [1,n_2]$ and for some integr $a_i\ge 0$, we have 
$$
|s_{[D,\sigma]}(y_i)|= 
\begin{cases}
a_im_2 +m_1+m_2+\sigma_i,  & \text{if $m_2$ is odd}, \\ 
a_i(m_2+1)+m_1+m_2+\sigma_i,  & \text{if $m_2$ is even}, 
\end{cases}
$$
where $\sigma_i\in [1,n_2]$ are all distinct. Since $n_2\le m_2$, 
for any two distinct $i,j\in [1,n_2]$, 
$$
s_{[D,\sigma]}(y_i)-s_{[D,\sigma]}(y_j)\not\equiv 
\begin{cases}
0 \pmod{m_2},  & \text{if $m_2$ is odd}, \\ 
0 \pmod{m_2+1},  & \text{if $m_2$ is even}. 
\end{cases}
$$
Consequently, $s_{[D,\sigma]}(y_i) \ne s_{[D,\sigma]}(y_j)$.  
\end{enumerate}
The proof for Case 1 is complete. 
	\medskip 

{\bf \noindent Case 2: $n_2\ge m_2+1$}. 

\medskip

In this case, we develop a result similar to Lemma~\ref{partitions}
but using nonconsecutive integers not necessarily starting at 1. 

For each $i\in [1,n_2-t_1]$, let $A_i$ be the set of all edges incident to $y_i$ in $H$. 
Clearly,  $|A_1|+|A_2|+\ldots+|A_{n_2-t_1}|=m_2$. 
Since $G$ has no vertex of degree 2 or isolated vertex,  $|A_i|\ge 2$. 
Let $m_2=3k+2\ell$, for some integers $k, \ell \ge 0$, 
where $k$ is the number of sets $A_i$'s  with an odd cardinality.  
We may assume that $k+\ell\ge 1$. Otherwise, we follow the same proof as in Case 1,
and the vertex sums at vertices from $T$ will naturally be all distinct since all these vertices have degree 1 in $G$. 

	\medskip 

{\bf \noindent Subcase 2.1: $k=0$}. 

\medskip
In this case, all $|A_i|$'s are even.
We give an orientation  $D$ of $G$ and a labeling $\sigma$ of $A(D)$ through four parts below. 
\begin{enumerate}[(1)]
	
	\item Orient and label $G[S]$:  applying Lemma~\ref{consecutive} to get the orientation and labeling with $p=0$ and $m=m_1$;
	\item Orient and label $H$: direct each edge from $T$ to $S$.
	 For each $i$,  
	partition all edges in $A_i$ into $|A_i|/2$ many 
	2-element subsets.  Thus, we have in total $m_2/2$ many 
	2-element subsets  $B_1, B_2, \ldots, B_{m_2/2}$ of edges. 
	For each  $B_i$, $i\in [1,m_2/2]$, 
	we assign 
	$$
	m_1+n_2+i, m-(i-1)
	$$ 
	to the two edges from it.
By the way above of assigning labels to edges in $A_i$'s, $i\in [1,n_2-t_1]$, the 
sum of labels assigned to edges from each $A_i$ is 
\begin{equation}\label{case21}
a_i (m+m_1+n_2+1) \quad \text{for some integer $a_i\ge 1$}. 
\end{equation}
	 
	\item Orient and label $M^*\setminus M=\{e_{n_1+1}, \ldots, e_{n_2}\}$: direct each edge from $T$ to $S$, 
	and for each $i\in [n_1+1, n_2]$, assign $m_1+(i-n_1)$ to $e_i$. 
	
	\medskip
	
	Let $D^*$ be the union of the digraphs obtained through the three  parts above, and $\sigma^*$ be the  labeling on $A(D^*)$ consists of the  three labelings above.
Assume that the sums at vertices from $S=\{x_1, \ldots, x_{n_1}\}$ satisfy 
	$$
	s_{[D^*,\sigma^*]}(x_1)\le	s_{[D^*,\sigma^*]}(x_2)\le \ldots \le s_{[D^*,\sigma^*]}(x_{n_1}). 
	$$
	
	\medskip

	\item Orient and label $M$: direct each edge from $T$ to $S$,   and for each $i\in [1, n_1]$, assign $m_1+n_2-n_1+i$ to $x_iy_i$.  
	
	Let $D$ and $\sigma$ be the resulting orientation  and labeling, respectively. 
	It is clear that $\sigma$ is injective. We show that $\sigma$ 
	is an antimagic labeling of $D$. 
	By Step 4, we have that 
	$$
	s_{[D,\sigma]}(x_1)<s_{[D,\sigma]}(x_2)<\ldots <s_{[D,\sigma]}(x_{n_1}). 
	$$
	Furthermore, for each $i\in [1,n_1]$, by Lemma~\ref{consecutive} and Step 1,   $s_{[D^*,\sigma^*]}(x_i)\ge \lfloor\frac{d_{G[S]}(x_i)-1}{2}\rfloor  -m_1$, we know $s_{[D,\sigma]}(x_i)\ge s_{[D^*,\sigma^*]}(x_i)+m_1+n_2-n_1+i\ge 0$. 
	For each vertex $y_i\in T$, $i\in [1,n_2]$, all the edges incident to $y_i$
	are oriented towards $S$. 	Thus, $s_{[D,\sigma]}(y_i)<0$.
	
	Thus, for each $x\in S$ and each $y\in T$, 
	$s_{[D,\sigma]}(x)>s_{[D,\sigma]}(y)$. 
	Therefore, it is left to only show that all vertices from $T$
	have distinct sums under $\sigma$ in $D$. 
	
	By Steps 2, 3 and 4, for each $i\in [1,n_2]$, we have 
	$$
	|s_{[D,\sigma]}(y_i)|= a_i(m_1+n_1+m+1)+m_1+\sigma_i \quad \text{for some integer $a_i\ge 1$}, 
	$$
	where $\sigma_i\in [1,n_2]$ are all distinct. Since $n_2< m_1+n_1+m+1$, 
	for any two distinct $i,j\in [1,n_2]$, 
	$$
	s_{[D,\sigma]}(y_i)-s_{[D,\sigma]}(y_j)\not\equiv 
	0 
	\pmod{m_1+n_1+m+1}
	$$
	Consequently, $s_{[D,\sigma]}(y_i) \ne s_{[D,\sigma]}(y_j)$.  
\end{enumerate}
The proof for Subcase 2.1 is complete.

\medskip 

{\bf \noindent Subcase 2.2: $k\ge 1$}. 

\medskip
Recall that  $m_2=3k+2\ell$ and $n_2\ge m_2+1$. 
Thus $m\ge n_2+m_2\ge  6k+4\ell+1 $, and 
$m-2k-\ell+2=m_2+n_2+m_1-2k-\ell+2>m_1+3k+\ell$. We assume, without loss of generality, that $|A_1|, \ldots, |A_k|$
are odd, and $|A_{k+1}|, \ldots, |A_{n_2-t_1}|$ are all even. 

We will use the labels from the set $A=[1,k]\cup [m_1+k+1, m_1+2k]\cup [m_1+3k+1,m_1+3k+\ell] \cup [m-2k-\ell+2, m-2k+1]\cup \{m-2k+2, m-2k+4, \ldots, m-2, m\}$ 
for edges from $A_i$'s. 
For each $i\in [1,k]$, edges in $A_i$ can be partitioned into 
one 3-subset, and $(|A_i|-3)/2$ many 2-subsets. 
For each $i\in [k+1, n_2-t_1]$, edges in $A_i$ can be partitioned into 
$|A_i|/2$ many 2-subsets. 
Let 
$B_1, B_2, \ldots, B_k$ be the $k$ 3-sets and $C_1, \ldots, C_\ell$ 
be the $\ell$ 2-sets  obtained by 
partition edges from each $A_i$'s. 
For each $i\in [1,k]$, we assign edges in each $B_i$
the following three numbers:
$$
i, \quad  m_1+k+i, \quad m-2i+2. 
$$

For each $i\in [1,\ell]$, we assign edges in each $C_i$
the following two  numbers:
$$
m_1+3k+i, \quad  m-2k+2-i. 
$$
By the way above of assigning labels to edges in $A_i$'s, $i\in [1,n_2-t_1]$, the 
sum of labels assigned to edges from each $A_i$ is 
\begin{equation}\label{case22}
a_i (m+m_1+k+2) \quad \text{for some integer $a_i\ge 1$}. 
\end{equation}

We give an orientation  $D$ of $G$ and a labeling $\sigma$ of $A(D)$ through four parts below. 
\begin{enumerate}[(1)]
	
	\item  Orient and label $G[S]$:  applying Lemma~\ref{consecutive} to get the orientation and labeling with $p=k$ and $m=m_1$;
	\item  Orient and label $H$: direct each edge from $T$ to $S$.
	Assign labels in the set $A$ to the edges in $\bigcup _{i=1}^{n_2-t_1}A_i$
	as described previously. 
	
Note that  the set of unused labels is $$B=[m_1+2k+1,m_1+3k]\cup[m_1+3k+\ell+1, m-2k-\ell+1]\cup \{ m-2k+3, m-2k+5, \ldots, m-1\},$$
	and $|B|=k+m-m_1-5k-2\ell+1+k-1=m_2+n_2-3k-2\ell=n_2$. 
	\item  Orient and label $M^*\setminus M=\{e_{n_1+1}, \ldots, e_{n_2}\}$: direct each edge from $T$ to $S$, 
	and assign the first $n_2-n_1$ smallest numbers from $B$ to 
	edges in  $M^*\setminus M$ such that distinct edges receive distinct labels. 
	
	\medskip
	
	Let $D^*$ be the union of the digraphs obtained through the three  parts above, and $\sigma^*$ be the  labeling on $A(D^*)$ consists of the  three labelings above. 
Assume that the sums at vertices from $S=\{x_1, \ldots, x_{n_1}\}$ satisfy 
	$$
	s_{[D^*,\sigma^*]}(x_1)\le	s_{[D^*,\sigma^*]}(x_2)\le \ldots \le s_{[D^*,\sigma^*]}(x_{n_1}). 
	$$
	
	\medskip

	\item  Orient and label $M$: direct each edge from $T$ to $S$,   and assign the 
	remaining  $n_2-(n_2-n_1)=n_1$ numbers  from $B$ 
	to edges in $M$ such that $x_iy_i$ is assigned with the 
	$i$-th smallest number. 
	
	Let $D$ and $\sigma$ be the resulting orientation  and labeling, respectively. 
	It is clear that $\sigma$ is injective. We show that $\sigma$ 
	is an antimagic labeling of $D$. 
	By Step 4, we have 
	$$
	s_{[D,\sigma]}(x_1)<s_{[D,\sigma]}(x_2)<\ldots <s_{[D,\sigma]}(x_{n_1}). 
	$$
	Furthermore, for each $i\in [1,n_1]$, by Lemma~\ref{consecutive} and Step 1,   $s_{[D^*,\sigma^*]}(x_i)\ge \lfloor\frac{d_{G[S]}(x_i)-1}{2}\rfloor  -m_1-k$, we know $s_{[D,\sigma]}(x_i)\ge s_{[D^*,\sigma^*]}(x_i)+m_1+2k+1\ge 0$. 
	For each vertex $y_i\in T$, $i\in [1,n_2]$, all the edges incident to $y_i$
	are oriented towards $S$. Thus, $s_{[D,\sigma]}(y_i)<0$. 
	
	Thus for each $x\in S$ and each $y\in T$, 
	$s_{[D,\sigma]}(x)>s_{[D,\sigma]}(y)$. 
	Therefore, it is left to only show that all vertices from $T$
	have distinct sums under $\sigma$ in $D$. 
	
	By Steps 2, 3, 4 and~\eqref{case22}, for each $i\in [1,n_2]$,  we have 
	$$
	|s_{[D,\sigma]}(y_i)|\ge  a_i(m+m_1+k+2)+\sigma_i \quad \text{for some integer $a_i\ge 0$}, 
	$$
	where $\sigma_i\in B$ are all distinct. Since $\sigma_i\le m-1< m+m_1+k+2$, 
	for any two distinct $i,j\in [1,n_2]$, 
	$$
	s_{[D,\sigma]}(y_i)-s_{[D,\sigma]}(y_j)\not\equiv 
	0 
	\pmod{m+m_1+k+2}
	$$
	Consequently, $s_{[D,\sigma]}(y_i) \ne s_{[D,\sigma]}(y_j)$.  
\end{enumerate}
The proof for Subcase 2.2 is complete.

\section{Proof of Theorem~\ref{th2}}

	Let $L$  be a spanning bipartite subgraph of $G$ with the maximum number of edges. Since $|E(L)|$ is maximum among all spanning bipartite subgraphs of $G$,
	$$
	d_L(v)\ge \frac{d_G(v)}{2} \quad \text{for every $v\in V(G)$}. 
	$$ 
	By Lemma~\ref{matching-in-bipartite}, we let $S\cup T$ 
	be a partition of $V(L)=V(G)$,  $M\subseteq E_L(S,T)$ 
	be a matching that saturates $S$ and has size $|S|$, and let  $L^*=L-E(L[S])$ 
	be the spanning  bipartite graph of $L$ between $S$ and $T$.  
	
	Let 
	$$
n_2=|S|, \quad 	n_1=|T|, \quad  S=\{x_1, x_2, \ldots, x_{n_2}\}, \quad 
	\quad   T=\{y_1, y_2, \ldots, y_{n_1}\}. 
	$$
	
	Assume, without loss of generality, that  
	$$M=\{x_1y_1, x_2y_2, \ldots, x_{n_1}y_{n_2}\}.$$
	For each $i\in [n_2+1, n_1]$, let $e_i$ be an edge incident to $y_i$ 
	in $L^*$, and let 
	$$
	M^*=M\cup \{e_{n_2+1}, \ldots, e_{n_1}\}.
	$$
	In other words,  each vertex from  $T$ is incident to one and exactly one edge in $M^*$.  
	Furthermore, let 
	$$
	H=L^*-M^*, \quad G_1=G-E(H)-M^*. 
	$$
	Note that for every vertex $y\in T$, 
	\begin{equation}\label{degreev}
	 d_H(v)=d_L(v)-1\ge  \frac{d_G(v)}{2}-1,  
	\end{equation}
and  $E(G_1)=E(G)\setminus (E(H) \cup M^*)=E(G[S])\cup E(G[T])\cup \big( E_G(S,T)\setminus E_{L^*}(S,T)\big)$. 
	We now modify $G_1$ to get a new graph by adding some edges from $H$
	such that in the new graph  the degree of every vertex from $T$ 
	is divisible by 4 and that every vertex from $T$ has a neighbor from $S$.  Specifically, 
	 for each $v\in T$, if $d_{G_1}(v) \equiv c \pmod{4}$, 
	 where $c=0,1,2,3$, we take exactly $4-c$ edges incident to $v$
	 in $H$ and add these $4-c$ edges into $G_1$.  Call $G_2$ the resulting in graph from $G_1$, and $H'$ the resulting in graph from $H$. 
	 From the construction, for each $v\in T$, 
		\begin{equation}\label{deg2}
	 d_{G_2}(v)\equiv 0 \pmod{4}, \quad d_{H'}(v)\ge d_{H}(v)+c-4, 
	 	\end{equation} 
	 where $c\in \{0,1,2,3\}$  satisfies $d_{G_1}(v) \equiv c \pmod{4}$. 
	 
	 We then split the bipartite graph $H'$ into two spanning subgraphs $H_1$
	 and $H_2$ of $H'$. For each $v\in T$, we let $A(v)$ be a set of $\frac{d_{G_2}(v)}{2}$ edges incident to $v$ in $H'$. Now let 
	 $$
	 V(H_2)=V(H'), \quad E(H_2)=\bigcup_{v\in T} A(v), \quad  H_1=H'-E(H_2). 
	 $$
	  From the construction and~\eqref{deg2}, for each $v\in T$, 
	 \begin{equation}\label{deg3}
	 d_{H_2}(v)=\frac{d_{G_2}(v)}{2}\equiv 0 \pmod{2}, \quad d_{H_1}(v)\ge d_{H}(v)+c-4-\frac{d_{G_2}(v)}{2}, 
	 \end{equation} 
	 where $c\in \{0,1,2,3\}$  satisfies $d_{G_1}(v) \equiv c \pmod{4}$. 
	 By~\eqref{degreev}, we have 
	 \begin{eqnarray}
	d_{H_1}(v)&\ge& d_{H}(v)+c-4-\frac{d_{G_2}(v)}{2} \nonumber \\\nonumber 
	&\ge & \lceil \frac{d_G(v)}{2}\rceil -1+c-4-\frac{d_{G_2}(v)}{2}\\\nonumber 
	&\ge & \lceil \frac{d_G(v)}{2}\rceil-1+c-4- \frac{1}{2}\big(\lfloor \frac{d_G(v)}{2} \rfloor +4-c\big)\\ \label{ge2}
	&\ge &  \frac{d_G(v)}{4}-7,  
	 \end{eqnarray}
	which is at least 2, since $\delta(G)\ge 33$. 
	 
	 Let $$m_1=e(H_1), \quad  m_2=e(G_2),\quad  m_3=e(H_2).$$  Note that $m_1+m_2+m_3+|M^*|=m:=e(G)$. 
	 We will now give an orientation  $D$ of $G$ and a labeling $\sigma$ of $A(D)$ through five  parts below. 
	 	 \begin{enumerate}[(1)]
	 	\item  Orient and label $H_1$: direct each edge from $T$ to $S$.
	 	For each $i\in [1,n_1]$, let $A_i$ be the set of all edges incident to $y_i$ in $H_1$. 
	 	Clearly, $|A_1|+|A_2|+\ldots+|A_{n_1}|=m_1$. 
	 	By~\eqref{ge2},  $|A_i|\ge 2$. 
	 	By Lemma~\ref{partitions} applied to $m_1$ with 
	 	$t=n_1$ and $r_i=|A_i|$ for each $i\in[1,t]$, the set $\{1,2,\ldots, m_1\}$
	 	can be partitioned into $R_1, R_2, \ldots, R_{n_1}$ such that for each $i\in [1,n_1]$, 
	 	$|R_i|=|A_i|$ and $\sum_{r\in R_i} r \equiv 0 \pmod{m_1+1}$ if $m_1$ is even, and 
	 	$\sum_{r\in R_i} r \equiv 0 \pmod{m_1}$ if $m_1$ is odd.  Label edges in $A_i$
	 	by integers in $R_i$ in an arbitrary way as long as distinct edges receive distinct labels. 
	 	\item  Orient and label $G_2$: Note that for each $y\in T$, $d_{G_2}(y)$ is even and $N_{G_2}(y)\cap S\ne \emptyset$
	 	by the construction of $G_2$. Thus, 
	 	we apply Lemma~\ref{consecutive}  to get the orientation and labeling of $G_2$ with $p=m_1$ and $m=m_2$ 
	 	with the furthermore requirement for vertices in $T$. Let $D_2$ be the orientation of $G_2$
	 	and $\sigma_2$ be the labeling. 
	 	We have  
	 	\begin{eqnarray}\label{sigma2}
	 		%-\lfloor \frac{d(v)-1}{2}\rfloor -(p+m)\le &  s(D,\sigma_1)(v)&\le -\lfloor \frac{d(v)-1}{2}\rfloor +(p+m), \quad \text{and }\\
	 		-(m_1+m_2)+\lfloor \frac{d_{G_2}(x)-1}{2}\rfloor\le &  s_{[D_2,\sigma_2]}(x)&\le \lfloor \frac{d_{G_2}(x)-1}{2}\rfloor +(m_1+m_2) \quad \text{for $x\in S$}, \nonumber \\
	 	&	s_{[D_2,\sigma_2]}(y) &= 	 \frac{d_{G_2}(y)}{2} \quad \text{for $y\in T$}. 
	 	\end{eqnarray}
 	
	 	\item  Orient and label $H_2$: applying Lemma~\ref{Teven} to get the orientation and labeling of $H_2$ with $p=m_1+m_2$ and $m=m_3$.
	 	Let $D_3$ be the orientation of $H_2$
	 	and $\sigma_3$ be the labeling. 
	 	We have   
	 	\begin{eqnarray}\label{sigma3}
	 		\lfloor \frac{d_{H_2}(x)-1}{2}\rfloor -\delta_m\le &  s_{[D_3,\sigma_3]}(x)&\le \lfloor \frac{d_{H_2}(x)-1}{2}\rfloor +\delta_m \quad \text{for each $x\in S$}, \nonumber\\
	 			& s_{[D_3,\sigma_3]}(y)&=
	 		-d_{H_2}(y) \quad \quad\quad \quad\,\,\,\,
	 		\text{for each $y\in T$,}
	 	\end{eqnarray}
	 	where $\delta_m =m_1+m_2+m_3$ if $m_3\equiv 0 \pmod{4}$, and $\delta_m =m_1+m_2+m_3+1$ if $m_3\equiv 2 \pmod{4}$. 
	 \item  Orient and label $M^*\setminus M=\{e_{n_2+1}, \ldots, e_{n_1}\}$: direct each edge from $T$ to $S$. 
	 If $m_3\equiv 0 \pmod{4}$, 
	  for each $i\in [n_2+1, n_1]$, assign $m_1+m_2+m_3+(i-n_2)$ to $e_i$. 
	  If $m_3\equiv 2 \pmod{4}$, assign $m_1+m_2+m_3$ to $e_{n_2+1}$, and 
	  for each $i\in [n_2+2, n_1]$, assign $m_1+m_2+m_3+(i-n_2)$ to $e_i$. 
	 	\medskip
	 
	 Let $D^*$ be the union of the digraphs obtained through the four  parts above, and $\sigma^*$ be the  labeling on $A(D^*)$ consists of the four labelings above. 
	 Assume that the sums at vertices from $S=\{x_1, \ldots, x_{n_2}\}$ satisfy 
	 $$
	 s_{[D^*,\sigma^*]}(x_1)\le	s_{[D^*,\sigma^*]}(x_2)\le \ldots \le s_{[D^*,\sigma^*]}(x_{n_2}). 
	 $$
	 
	 \medskip

	 	\item  Orient and label $M$: direct each edge from $T$ to $S$. If $n_1\ge n_2+1$ or  $m_3\equiv 0 \pmod{4}$, 
	 	    for each $i\in [1, n_2]$, assign $m_1+m_2+m_3+n_1-n_2+i$ to $x_iy_i$.  
	 	    If $n_1=n_2$ and  $m_3\equiv 2 \pmod{4}$, assign $m_1+m_2+m_3$ to $x_1y_1$, and for 
	 	    each $i\in [2, n_2]$, assign $m_1+m_2+m_3+i$ to $x_iy_i$.

	 \end{enumerate}
 
 Let $D$ and $\sigma$ be the resulting orientation  and labeling, respectively. 
 It is clear that $\sigma$ is injective. We show that $\sigma$ 
 is an antimagic labeling of $D$. 
 
 	By Step 5, we have 
 $$
 s_{[D,\sigma]}(x_1)<s_{[D,\sigma]}(x_2)<\ldots <s_{[D,\sigma]}(x_{n_2}). 
 $$
 Furthermore, for each $i\in [1,n_2]$, by \eqref{sigma2} and \eqref{sigma3},   $s_{[D^*,\sigma^*]}(x_i)\ge \lceil\frac{d_{G_2}(x_i)-1}{2}\rceil -m_1-m_2+\lceil\frac{d_{H_2}(x_i)-1}{2}\rceil-m_1-m_2-m_3-1$, 
  we know $s_{[D,\sigma]}(x_i)\ge s_{[D^*,\sigma^*]}(x_i)+m_1+m_2+m_3\ge -m_1-m_2-1$. 
 For each vertex $y_i\in T$, $i\in [1,n_1]$, for  all the edges incident to $y_i$
 that are contained in $G_2\cup H_2$, the partial sum  at $y_i$ of the labels assigned to these 
 edges is zero by~\eqref{deg3}, \eqref{sigma2} and \eqref{sigma3}. All other edges incident to $y_i$ that are contained in 
 $H_1\cup M^*$
 are oriented towards $S$. Thus, $s_{[D,\sigma]}(y_i)<0$. 
 Furthermore, by Steps 1, 4 and 5, 
 $s_{[D,\sigma]}(y_i)\le-m_1-m_2-m_3-3$.  
 
 Thus, for each $x\in S$ and each $y\in T$, 
 $s_{[D,\sigma]}(x)>s_{[D,\sigma]}(y)$. 
 Therefore, it is left to only show that all vertices from $T$
 have distinct sums under $\sigma$ in $D$. 
 
 By Steps 1, 4, 5, and \eqref{sigma2} and \eqref{sigma3}, for each $i\in [1,n_1]$ and  some integer $a_i\ge 1$, we have  
 $$
 |s_{[D,\sigma]}(y_i)|= 
 \begin{cases}
 \frac{d_{G_2}(y_i)}{2}-d_{H_2}(y_i)+a_im_1 +m_1+m_2+m_3+\sigma_i,  & \text{if $m_1$ is odd}, \\ 
 \frac{d_{G_2}(y_i)}{2}-d_{H_2}(y_i)+a_i(m_1+1) +m_1+m_2+m_3+\sigma_i,  & \text{if $m_1$ is even}, 
 \end{cases}
 $$
 where $\sigma_i\in [1,n_1]$ are all distinct, 
 and $ \frac{d_{G_2}(y_i)}{2}-d_{H_2}(y_i)=0$. 
  Since $ m_1\ge 2n_1>n_1$ by~\eqref{ge2}, 
 for any two distinct $i,j\in [1,n_1]$, 
 $$
 s_{[D,\sigma]}(y_i)-s_{[D,\sigma]}(y_j)\not\equiv 
 \begin{cases}
 0 \pmod{m_1},  & \text{if $m_1$ is odd}, \\ 
 0 \pmod{m_1+1},  & \text{if $m_1$ is even}. 
 \end{cases}
 $$
 Consequently, $s_{[D,\sigma]}(y_i) \ne s_{[D,\sigma]}(y_j)$.  

The proof  is now complete.

%	 Applying Lemma~\ref{partitions} to $H_1$. 
%	 	 Applying Lemma 6 to $G_2$ with $p=m_1$ and $m=m_2$,
%	 applying Lemma 7 to  $H_2$ with $p=m_1+m_2$ and $m=m_3$. 

%%%%%%%%%%%%%%%%%%%%%%%%%%%%%%%%%%%%%%%%%%%%%%%%%%%%%%%
 \bibliographystyle{plain}
\bibliography{SSL-BIB}

\end{document}